\documentclass[ a4paper, 11pt]{article}

\usepackage[latin1]{inputenc}
\usepackage[T1]{fontenc}
\usepackage{amssymb}
\usepackage{amsmath}
\usepackage{amsthm}
\usepackage[english]{babel}
\usepackage{graphicx}
\usepackage{epsfig}
\usepackage{color}
\usepackage{latexsym}
\usepackage[all]{xy}
\usepackage{geometry}

\theoremstyle{plain}
\newtheorem{thm}{Theorem}

\newtheorem{lem}[thm]{Lemma}

\theoremstyle{definition}
\newtheorem{definition}[thm]{Definition}

\DeclareMathOperator{\Rips}{Rips}
\DeclareMathOperator{\Homo}{H}

\title{A subset of Euclidean space with large Vietoris-Rips homology}

\author{Jean-Marie Droz}

\date{\today}

\begin{document}

\maketitle


\section{Introduction}

In this article, we construct a compact subset  $K$ of the four dimensional Euclidean space with the following property: For all values of the parameter $a$ in an interval, the Vietoris-Rips complex $\Rips(K,a)$ has first homology $\Homo_1(\Rips(K,a))$ uncountable. 
 To fix notation, we state the definition of the Vietoris-Rips complex. This answers a question that arose in work on persistent homology published in \cite{chazal} in a discussion between authors of \cite{chazal} and S. Smale.

\section{Acknowledgements}
We would like to thank Fr\'ed\'eric Chazal for introducing us to the question treated here and for helpful discussions.

\section{Main theorem}

\begin{definition}
Given a pseudometric space $(M,d)$ and a number $a\in \mathbb{R}^+$, we define the {\it Vietoris-Rips complex} of $(M,d)$ at $a$, $\Rips(M,a)$, as the simplicial complex with, as $k$-simplexes (for $k\in \mathbb{N}$), the sets of $k+1$ distinct points of $M$ with diameter $\leq a$.
\end{definition}

\begin{definition}
A function $f:A \rightarrow B$ between pseudometric spaces $(A,\delta_A)$ and $(B,\delta_B)$ is {\it $C$-close-expanding} if $\forall x,y\; \delta_B(f(x),f(y))\geq C\cdot\sqrt{\delta_A(x,y)}$. 
\end{definition}

\begin{definition}
We define the distance $\delta_3: \{0,1,2\}^\mathbb{N}\times\{0,1,2\}^\mathbb{N}\rightarrow \mathbb{R}$ between two ternary sequences as $\delta_3((s_i)_{i\geq 0},(t_j)_{j\geq 0})=3^{-\min \{k| s_k\neq t_k\}}$. We call {\it ternary ultrametric distance} the function $\delta_3$.
\end{definition}


\begin{definition}
Two pseudometrics $d_1$ and $d_2$ over a space $S$ are called equivalent when there are real constants $c_1,c_2>0$ such that $\forall x,y\in S$ $c_1\cdot d_1(x,y)\geq d_2(x,y)\geq c_2\cdot d_1(x,y)$
\end{definition}

We denote by $I_{many}$ the product of the interval $[0,1]$ with a set $D$ of cardinality $2^{\aleph_0}$. We equip $I_{many}$ with the pseudometric induced by the absolute value metric on the interval. 
 
\begin{lem}
\label{existenceCloseExpanding}
There is a $\frac{1}{243}$-close-expanding injective function  $e:I_{many}\rightarrow [0,1]^3$, where the unit cube is equipped with the Euclidean metric. 
\end{lem}
\begin{proof}
We consider points in $I_{many}$ as pairs of a ternary sequence and a binary sequence, the ternary sequence corresponding to the expansion in base $3$ of a number from the unit interval and the binary sequence to an element of $D$. If a number has two representations in base $3$ we simply choose one arbitrarily. In the following, we will often ignore the distinction between a number in the unit interval and the sequence representing it in base $3$.

For $i\in\{0,1,2\}$, let $e_i$ be the function $I_{many}\rightarrow [0,1]$ mapping $p=((t_0,t_1,\ldots),(b_0,b_1,\ldots))\in I_{many}$ to the number represented by the sequence $$(t_{2i},t_{2i+1},b_0, \ldots ,t_{6k+2i},t_{6k+2i+1},b_k,t_{6(k+1)+2i},t_{6(k+1)+2i+1},b_{(k+1)}, \ldots ).$$ We define the function $e$ as mapping an element $p\in I_{many}$ to $(e_0(p),e_1(p),e_2(p))\in [0,1]^3$.

It remains to prove that $e$ has the properties claimed. We begin by observing that in the images of $e_i$, the absolute value metric is equivalent to the natural ultrametric. This is implied by the fact that the representation in base $3$ of coordinates of points in the images cannot have a $2$ at places numbered $3k+2$ for $k\in \mathbb{N}$. (The image look like a product of asymmetric versions of the Cantor set.) The construction then straightforwardly implies that $e$ is injective, because no ``information'' is lost going from $p=((t_0,t_1,\ldots),(b_0,b_1,\ldots))$ to $e(p)$, all the $t_k$ and $b_k$ appearing somewhere in the sequences corresponding to $(e_0(p),e_1(p),e_2(p))$.

To prove that $e$ is close-expanding, we must show that if $|x-y|>\epsilon$, then $\|e(x)-e(y)\|_2>c\cdot \sqrt{\epsilon}$ for a fixed constant $c>0$. This is a consequence of four facts.
\begin{enumerate}
\item For $x,y\in[0,1]$, $|x-y|\leq \delta_3(x',y')$ holds, where $x'$ and $y'$ are representations in base $3$ of $x$ and $y$.
\item{ For $x,y\in \{0,1,2\}^\mathbb{N}$ and $x',y'\in I_{many}$ projecting to $x$ and $y$, $$\sqrt{\delta_3(x,y)}\leq 3\cdot \max(\delta_3(e_0(x'),e_0(y')),\delta_3(e_1(x'),e_1(y')),\delta_3(e_2(x'),e_2(y'))).$$}
\item On the image of $e$, the maximum metric is equivalent to the maximum of the ternary ultrametric on the coordinates. In particular, we have: 
$$\|e(x)-e(y)\|_\infty\geq 3^{-4}\cdot\max(\delta_3(e_0(x),e_0(y)),\delta_3(e_1(x),e_1(y)),\delta_3(e_2(x),e_2(y))).$$
\item The maximum metric is a lower bound for the Euclidean metric.

\end{enumerate}
The first fact is a direct consequence of the definitions. Statement number two is established by observing that if $x$ and $y$ differ at the $n$-th place, then, for some $i$, $e_i(x)$ and $e_i(y)$ will differ at worse at the $\lfloor n/2+1\rfloor$-th place. The third inequality is implied by the impossibility for the expansion in base $3$ of coordinates of points in the image of $e$ to have a $2$ at places numbered $3k+2$ for $k \in \mathbb{N}$. Number 4 is well-known. 

Combining the four inequalities sequentially, we obtain $\|e(x)-e(y)\|_2>3^{-5}\cdot \sqrt{\epsilon}$, so that $e$ is close expanding with constant $\frac{1}{243}$.
\end{proof}

\begin{lem}
\label{circleinsideparabola}
For $x\in [-r,r]$ and $r \in (0,\infty)$, $r-\sqrt{r^2-x^2}\geq \frac{x^2}{2r}$. In other words, a circle of radius $r$, tangent at $0$ with the parabola of equation $\frac{x^2}{2r}$ is completely above the parabola.
\end{lem}

Let $T$ be the set $\{(\frac{x}{2\cdot243^2},e((x,y)))| x\in [0,1],y\in D\}$ and $\overline{T}$ its closure, which is bounded and hence compact. 
\begin{thm}
\label{counterexample}
The compact subset  $K=\overline{T}\cup \{0\}\times[0,1]^3 \cup \{1\}\times[0,1]^3$ considered as a subset of the four dimensional Euclidean space has $\Homo_1(\Rips(K,a))$ uncountably generated for $a\in [1-\frac{1}{2\cdot 243^2},1]$.
\end{thm}
\begin{proof}
Let us fix an arbitrary $a\in [1-\frac{1}{2\cdot 243^2},1]$. A 1-simplex of length $a$ in $\Rips(K,a)$ between a point of $\{1\}\times[0,1]^3$ and a point of $\overline{T}$ will be called a rigid 1-simplex. 

Let $X\in \{1\}\times[0,1]^3$ and $Y\in \overline{T}$ be the two endpoints of a rigid 1-simplex. Since the distance from $Y$ to $X$ is equal to the distance from $Y$ to $\{1\}\times[0,1]^3$, there is only one line segment of length inferior or equal to $a$ between $Y$ and $\{1\}\times[0,1]^3$.

 Lemma \ref{circleinsideparabola} shows that if there is another point $Y'\in \overline{T}$ at distance less or equal to $a$ from $X$, the relation $\epsilon> \frac{l^2}{2}$ would hold between the length $\epsilon$ of the projection on the first coordinate of the vector $\vec{YY'}$ and the length $l$ of the projection of $\vec{YY'}$ on the $\{1\}\times[0,1]^3$ plane.  We would then have $\sqrt{2\cdot\epsilon}> l$, contradicting that the function $e$ is $\frac{1}{243}$-close-expanding and therefore $l\geq \frac{1}{243} \sqrt{2\cdot 243^2 \epsilon}$. Thus, a rigid 1-simplex is not contained in any non-trivial 2-simplex of $\Rips(K,a)$. 

For each $a$, we have $2^\omega$ rigid 1-simplexes coming from the choice of a point of $D$ to generate the points in $\overline{T}$ of first coordinate $1-a$. Moreover, any two rigid 1-simplex can be completed to a 1-cycle of $\Rips(K,a)$, using 1-simplexes of small length. This implies that $\Homo_1(\Rips(K,a))$ has uncountable rank. 

The set $\Rips(K,a)$ is compact because by Lemma \ref{existenceCloseExpanding}, it is a union of finitely many compact sets.
\end{proof}

\end{document}